\documentclass{article}
\usepackage{geometry}
\usepackage{graphicx}	
\usepackage[cp1251]{inputenc}
\usepackage[english]{babel}
\usepackage{mathtools}
\usepackage{amsfonts,amssymb,mathrsfs,amscd,amsmath,amsthm}
\usepackage{verbatim}

\def\thtext#1{
  \catcode`@=11
  \gdef\@thmcountersep{. #1}
  \catcode`@=12
}

\def\threst{
  \catcode`@=11
  \gdef\@thmcountersep{.}
  \catcode`@=12
}

\theoremstyle{plain}
\newtheorem{thm}{Theorem}
\newtheorem{prop}{Proposition}[section]
\newtheorem{cor}[prop]{Corollary}
\newtheorem{ass}[prop]{Assertion}
\newtheorem{lem}[prop]{Lemma}

\theoremstyle{definition}
\newtheorem{dfn}[prop]{Definition}

 \catcode`@=11
 \def\.{.\spacefactor\@m}
 \catcode`@=12

\newcommand{\e}{\varepsilon}

\newcommand{\g}{\gamma}
\newcommand{\G}{\Gamma}

\newcommand{\R}{\mathbb{R}}
\newcommand{\Z}{\mathbb{Z}}

\newcommand{\rom}[1]{{\em #1}}

\renewcommand{\:}{\colon}

\renewcommand{\c}{\circ}
\renewcommand{\d}{\partial}

\newcommand{\oPi}{\stackrel{\raise-2pt\hbox{$\c$}}\Pi}
\newcommand{\oW}{\stackrel{\raise-2pt\hbox{$\c$}}W}
\newcommand{\sm}{\setminus}

\renewcommand{\ss}{\subset}
\newcommand{\x}{\times}

\newcommand{\mst}{{\operatorname{mst}}}

\newcommand{\smt}{{\operatorname{smt}}}

\newcommand{\sr}{\operatorname{sr}}

\begin{document}
 \title{Branched Coverings and Steiner Ratio}
 \author{A.~O.~Ivanov, A.~A.~Tuzhilin\\
\small  Lomonosov Moscow State University, Faculty of Mechanics and Mathematics}
 \date{}
 \maketitle

 \begin{abstract}
For a branched locally isometric covering of metric spaces with intrinsic metrics, it is proved that the Steiner ratio of the base is not less than the Steiner ratio of the total space of the covering. As applications, it is shown that the Steiner ratio of the surface of an isosceles tetrahedron is equal to the Steiner ratio of the Euclidean plane, and that the Steiner ratio of a flat cone with angle of $2\pi/k$ at its vertex is also equal to the Steiner ratio of the Euclidean plane.
\end{abstract}

\section {Introduction}
\markright {\thesection.~Introduction}
In the wide spectrum of Optimal Connection Problems, estimations and calculations of Steiner ratio of a metric space play a special part. Recall necessary definitions and facts, see details in~\cite{ITbookWP} or~\cite{ITJar}. Let  $X$ be a metric space with a distance function $\rho$. For each its finite subset $M$ there exists a \emph{minimal spanning tree}, i.e., a tree of the least possible weight among all the trees with vertex set $M$, where the weight of an edge $xy$ is defined as the distance $\rho(x,y)$ in $X$ between the corresponding vertices. By $\mst_X(M)$ we denote the length of a minimal spanning tree for $M\subset X$. From the times of C.~Gauss it is well-known that sometimes one can  connect the same set $M$ by a shorter tree permitting additional vertices. The value 
$$
\smt_X(M)=\inf_{N:M\subset N}\bigl\{\mst_X(N)\bigr\},
$$
where the infimum is taken over all finite subsets $N$ of the space $X$ containing $M$, is called the \emph{length of a shortest tree for the set $M$}. If this infimum is attained at some set $N$, then the corresponding tree is referred as a  \emph{shortest tree\/} or a  \emph{Steiner minimal tree\/} for $M$.

The problem of a shortest tree finding for a given finite subset of a metric space (so called  \emph{generalized Steiner problem\/}) is very complicated from the computational point of view, see~\cite{GGJ}. Therefore, in practice some heuristic algorithms are usually applied to solve it, and the most frequent of those heuristics is a minimal spanning tree construction (notice that a minimal spanning tree can be constructed in a polynomial time). Steiner ratio is used as a measure of a relative error of such approximation in the worst possible situation. It was introduced by E.\,N.~Gilbert and H.\,O.~Pollack~\cite{GilPol} in the case of Euclidean plane. In general case it can be defined s follows.

\begin{dfn}
The number
 $$
\sr_X(M)=\frac{\smt_X(M)}{\mst_X(M)}
 $$
is called the {\em Steiner ratio of a finite subset $M$ of a metric space $X$}. The {\em Steiner ratio $\sr(X)$ of a metric space $X$} is defined as the following value:
 $$
\sr(X)=\inf_{M:M\ss X}\sr_X(M),
 $$
where the infimum is take over all finite subsets $M$ of the set  $X$, consisting of at least two elements.
\end{dfn}

In the same paper~\cite{GilPol}, E.\,N.~Gilbert and H.\,O.~Pollack conjectured that the Steiner ratio of Euclidean plane is attained at the vertex set of a regular triangle and, hence, is equal to $\sqrt 3/2$. But this conjecture is not proved yet, despite many attempts of many authors. The most famous among those efforts is the paper of D.\,Z.~Du and F.~Hwang~\cite{DuHwang}, see details in~\cite{ITLup}, \cite{ITAlg}, and~\cite{IT22}.

At present, an exact value of the Steiner ratio is calculated for the Manhattan plane~\cite{Hw}, Lobachevski plane~\cite{IK}, and Hadamard  spaces (i.e., simply connected A.\,D.~Alexandrov spaces) of negative curvature~\cite{Zav}. Also, there are several estimates and theorems describing some properties of this interesting characteristic of metric spaces, see a review in~\cite{Cies} and~\cite{CiesRep}. Also, recently some analogues of the Steiner ratio related to the concept of a minimal filling of a finite metric space were introduced~\cite{ITZapMS}. That stimulates new research activities, see, for example, the results of A.~Pakhomova on continuity and discontinuity of the ratios~\cite{Pakh}, estimates of Z.~Ovsyannikov~\cite{OvsComp} concerning spaces of compacts, and also a recent  review~\cite{ITZapAMS}.

The aim of the present paper is to generalize the results of the paper~\cite{CiesIT} on the relations between the Steiner ratio of the base and the one of the total space of a locally isometric covering of Riemannian manifolds, see Theorem~\ref{th:bundle},  to the case of metric spaces with an intrinsic metric and branched coverings, see Theorem~\ref{thm:branch_bundle}.  As applications, it is shown that the Steiner ratio of the surface of a triangular pyramid with equal faces is equal to the Steiner ratio of the Euclidean plane, and the Steiner ratio of a flat cone with angle of $2\pi/k$ at its vertex is also equal to the Steiner ratio of the Euclidean plane (Corollaries~\ref{cor:2-dim_tetr} and~\ref{cor:2-dim_cones}).

\section {Preliminaries}
\markright{\thesection.~Preliminaries}

It is well-known that $1/2\le \sr(X)\le 1$ for any metric space $X$. We also need the following simple result that can be found, for example in~\cite{Cies}.

\begin{ass}\label{ass:sub}
Let $Y$ be a subset of a metric space $X$ endowed with the metric obtained as restriction of the initial distance function from $X$. Then the Steiner ratio of the metric space $Y$ is not less than the Steiner ratio of $X$.
\end{ass}

Continuity of the Steiner ratio $\sr_X(M)$ as a function on the elements of the set $M$ follows from the continuity of the length of minimal spanning tree and the length of a shortest tree. Namely, the following result holds.

\begin{ass}\label{ass:cont}
Let  $M=\{m_1,\ldots,m_n\}$ be a finite subset of a metric space $X$. The Steiner ratio $\sr_X(M)$ depends continuously on $M$, i.e., the function $f(m_1,\ldots,m_n)=\sr_X(M)$ is continuous on $X^n$.
\end{ass}

Recall the main results of paper~\cite{CiesIT}.

\begin{thm}\label{thm:riem}
The Steiner ratio of an arbitrary connected $n$-dimensional Riemannian manifold does not exceed the Steiner ratio of the Euclidean space  $\R^n$.
\end{thm}

Recall that a continuous mapping $f\:X\to Y$ of path connected topological spaces is called a \emph{covering}, if each point $y\in Y$ possesses a neighborhood $V$ such that its complete pre-image $f^{-1}(V)$ is homeomorphic to $V\x S$, where  $S$ is some fixed set with discrete topology. Connected components of the complete pre-image $f^{-1}(V)$ are called \emph{leaves over $V$},  and the complete pre-image $f^{-1}(y)$ is called the \emph{fiber over $y$}. A neighborhood $V$ is referred as a  \emph{normal neighborhood of the point $y$}. A neighborhood  $U$ of a point  $x\in X$ is called  \emph{normal}, if its image $f(U)$ is a normal neighborhood of the point $f(x)$, and the restriction of  $f$ onto $U$ is a homeomorphism. It is clear that each normal neighborhood $U$ of the point $x$ is a leave over its $f$-image.

Let  $X$ and $Y$ be metric spaces. A covering $f\:X\to Y$ is called \emph{locally isometric}, if each point $x\in X$ has a neighborhood such that the restriction of $f$ onto it is an isometry. The next statement is evident.

\begin{ass}\label{ass:cover_prop}
Let  $f\:X\to Y$ be a locally isometric covering of path connected metric spaces. Then 
\begin{itemize}
\item the mapping  $f$ preserves the length of all measurable curves, namely, if $\g\:[a,b]\to X$ is a measurable curve, then the length of the curve $\g$ is equal to the length of the curve $\g\c f$\rom;
\item if the metrics on  $X$ and $Y$ are intrinsic, then the mapping $f$ does not increase the distances between points.
\end{itemize}
\end{ass}

\begin{thm}[see~\cite{CiesIT}]\label{th:bundle}
Let $W$ and $M$ be path connected Riemannian manifolds and $\pi\:W\to M$ be a locally isometric covering. Then the Steiner ratio of the base $M$ is not less than the Steiner ratio of the total space $W$.
\end{thm}

\section {Main Result}
\markright {\thesection.~Main Result}
Let  $f\:X\to Y$ be a continuous mapping of path connected topological spaces,  and $B\subset Y$ be a discrete subset of $Y$. Assume that the complete pre-image  $A=f^{-1}(B)$ of the set $B$ is also a discrete subset in $X$ and $f\:X\setminus A\to Y\setminus B$ is a covering. Then $f$ is called a  \emph{branched covering}, $Y$ is referred as the \emph{base} and $X$ is referred as the  \emph{total space\/} of the covering, $A$ is called the \emph{set of singular points}, and $B$ is called the \emph{set of singular values\/} of the covering. If $X$ and $Y$ are metric spaces, and $f\:X\setminus A\to Y\setminus B$ is a locally isometric covering, then the branched covering $f$ is also called \emph{locally isometric}.

Let us state the main result of the paper.

\begin{thm}\label{thm:branch_bundle}
Let $X$ and $Y$ be path connected metric spaces with intrinsic metrics, and let $f\:X\to Y$ be a locally isometric branched covering. Then the Steiner ratio of the base $Y$ is not less than the Steiner ratio of the total space $X$.
\end{thm}

We prove Theorem~\ref{thm:branch_bundle} in two steps. At first we reduce the problem to non-branched coverings, and then complete the proof considering the case of non-branched coverings.

\subsection{Reduction to Non-Branched Coverings Case}

We need the following Lemma.

\begin{lem}\label{lem:dence}
Let $Y$ be an everywhere dense subset of a metric space $X$. Endow $Y$ with a metric obtained as the restriction of the initial distance function defined on $X$. Then for any finite subset $M\subset Y$ the lengths of shortest trees connecting $M$ in $X$ and in $Y$ coincide with each other. Therefore, for any finite non-single point $M$ the equality $\sr_X(M)=\sr_Y(M)$ holds.
\end{lem}

\begin{proof}
The inequality $\smt_X(M)\le\smt_Y(M)$ is evident from the definition of the length of a shortest tree. The inverse inequality follows from the continuity of the distance between points. Indeed, for any positive $\e$ there exists an $N\subset X$ with $M\subset N$, such that  $\mst(N)\le\smt_X(M)+\e$. Since $Y$ is an everywhere dense subset of $X$,  then there exists a subset $N'\subset Y$, $M\subset N'$, such that $\bigl|\mst(N') -\mst(N)\bigr|<\e$, therefore
$$
\smt_Y(M)\le\mst(N')\le\mst(N)+\e\le\smt_X(M) +2\e,
$$
that implies the required inequality due to arbitrariness of $\e$.
\end{proof}

\begin{ass}\label{ass:dence}
Let  $Y$ be an everywhere dense subset of a metric space $X$. Endow $Y$ with a metric obtained as the restriction of the initial distance function defined on  $X$. Then  $\sr(X)=\sr(Y)$.
\end{ass}

\begin{proof}
Indeed, Assertion~\ref{ass:sub} implies that $\sr(X)\le\sr(Y)$. To prove the inverse inequality, let us fix an arbitrary positive $\e$  and a finite non-single point set $M\subset X$ such that  $\sr_X(M)\le\sr(X)+\e$. Due to Assertion~\ref{ass:cont}, there exists a finite non-single point set $N\subset Y$ such that $\bigl|\sr_X(N)-\sr_X(M)\bigr|<\e$. The latter inequality together with Lemma~\ref{lem:dence} implies that 
$$
\sr(Y)\le\sr_Y(N)=\sr_X(N)\le\sr_X(M)+\e\le\sr(X)+2\e,
$$
Now the inequality required follows from the arbitrariness of  $\e$.
\end{proof}

\begin{cor}
Let  $X$ and $Y$ be path connected metric spaces with intrinsic metrics, $f\:X\to Y$ be a locally isometric branched covering, $A\subset X$ be the set of singular points of $f$, and $B\subset Y$ be the set of singular values of $f$. Then Theorem~$\ref{thm:branch_bundle}$ is valid for the branched covering $f\:X\to Y$, if and only if it is valid for the non-branched covering $f\:X\sm A\to Y\sm B$.
\end{cor}

\subsection{Non-Branched Coverings Case}
Recall that a  \emph{topological graph\/} is defined as a finite one-dimensional cell complex, i.e., a topological space obtained from a finite set of segments by identification of some their end-points. Zero-dimensional cells of a topological graph are referred as its  \emph{vertices}, and its one-dimensional cells are called \emph{edges\/} of the graph. Since we are interested in boundary problems, then we always assume that each topological graph $G$ under consideration possesses a fixed subset $\d G$ of its vertex set referred as a \emph{boundary\/} of $G$.

A \emph{network\/} in a topological space is a continuous mapping of a connected topological graph into this space. We say that a network $\G\:G\to X$ \emph{connects $M\ss X$}, if  $M=\G(\d G)$.

For a metric space $X$ with strictly intrinsic metric the value $\smt_X(M)$ can be defined in terms of networks in the following way. At first, we define the \emph{length $|\G|$ of a network $\G$} as the sum of the lengths of all its edges--curves. Then we consider all networks $\G$ connecting $M$ and put $\smt_X(M)=\inf_{\G}|\G|$. It is easy to see that the same value $\smt_X(M)$ can be obtained by considering only trees (i.e., connected acyclic graphs).

We need the following mappings lifting theorem, see, for example,~\cite{FomFuks}. We state it in a particular case which is important for us.

\begin{ass}\label{ass:lift}
Let $f\:X\to Y$ be a non-branched covering, $G$ be a topological tree, and $\G_Y\:G\to Y$ be a network. Then there exists a lifting network $\G_Y\:G\to X$ such that  $\G_X=f\c\G_Y$.
\end{ass}

Thus, it remains to prove the following result.

\begin{thm}\label{thm:nonbranch_bundle}
Let  $X$ and $Y$ be path connected metric spaces with intrinsic metrics, and $f\:X\to Y$ be a locally isometric non-branched covering. Then the Steiner ratio of the base $Y$ is not less than the Steiner ratio of the total space $X$.
\end{thm}

 \begin{proof}
Let  $N$ be an arbitrary finite non-single point subset in $Y$, and $\e$ be an arbitrary positive. Show that there exists a finite set $M\subset X$ such that  $f(M)=N$ and $\smt_X(M)\le(1+\e)\,\smt_Y(N)$. Indeed, by definition of $\smt_Y(N)$, there exists a measurable network $\G_Y\:G\to Y$ parameterized by some topological tree $G$ such that $N=\G_Y(\d G)$ and $|\G_Y|\le(1+\e)\,\smt_Y(N)$. Assertion~\ref{ass:lift} implies that there exists a lifting $\G_X\:G\to X$. Put $M=\G_X(\d G)$. By definition of the lifting, we have $f(M)=N$. Due to Assertion~\ref{ass:cover_prop}, the lengths of the networks $\G_X$ and $\G_Y$ coincide with each other, therefore,  $\smt_X(M)\le|\G_X|=|\G_Y|\le(1+\e)\,\smt_Y(N)$, and hence, the set $M$ matches our claim.

On the other hand, Assertion~\ref{ass:cover_prop} implies that for any finite set $M\subset X$ the inequality $\mst_X(M)\ge \mst_Y\bigl(f(M)\bigr)=\mst_Y(N)$ is valid. So, 
$$
\sr_X(M)=\frac{\smt_X(M)}{\mst_X(M)}\le\frac{(1+\e)\smt_Y(N)}{\mst_Y(N)}=(1+\e)\,\sr_Y(N)\le\sr_Y(N)+\e.
$$
Thus, for any $\e>0$ and any finite $N\subset Y$ consisting of at least two points, we have constructed a set $M\ss X$ such that $\sr_X(M)\le\sr_Y(N)+\e$, and hence $\sr(X)\le\sr(Y)+\e$. Since $\e$ is arbitrary, then the required inequality is valid.
\end{proof}

\begin{cor}\label{cor:cover-e}
Let  $Y$ be a connected Riemannian manifold with isolated singularities, $f\:\R^n\to Y$ be a locally isometric branched covering, and the set of singularities of $Y$ is contained in the set of singular values of the mapping $f$. Then $\sr(Y)=\sr(\R^n)$.
\end{cor}

\begin{proof}
By $A$ we denote the set of singular points of $f$, and by $B$ we denote the set of singular values of  $f$. Then, due to Assertion~\ref{ass:dence}, we have  $\sr(Y)=\sr(Y\sm B)$. Due to assumptions,  $Y\sm B$ is a Riemannian manifold  (without singularities), therefore, due to Theorem~\ref{thm:riem}, we conclude that $\sr(Y\sm B)\le\sr(\R^n)$. On the other hand, Theorem~\ref{thm:branch_bundle} implies that $\sr(Y)\ge\sr(\R^n)$, that completes the proof.
\end{proof}

\section {Examples: Polyhedra and Cones}
\markright {\thesection.~Examples: Polyhedra and Cones}
In this Section we give several examples of calculation and estimation of the Steiner ratio obtained by application of the technique elaborated above.

A concept of a multidimensional polyhedron (or a polyhedral surface) is well-known. Corresponding definitions can be found, for example, in books~\cite{AD} and~\cite{Zieg}. In what follows we just need to mention that from the metric geometry point of view an $n$-dimensional polyhedron is an $n$-dimensional Riemannian manifold with isolated singular points (the \emph{vertices\/} of the polyhedron) and with Euclidean metric. 

Assertion~\ref{ass:dence} and Theorem~\ref{thm:riem} immediately implies the following result.

\begin{cor}
The Steiner ratio of an arbitrary $n$-dimensional polyhedron does not exceed $\sr(\R^n)$.
\end{cor}

\begin{ass}
If a two-dimensional polyhedron contains a vertex, whose total angle, i.e., the sum of all adjacent flat angles of the corresponding faces,  is more than $2\pi$, then its Steiner ratio is strictly less than the one of the Euclidean plane.
\end{ass}

\begin{proof}
Consider an equilateral triangle inscribed into a circle of a small radius $r$ centered at such a vertex. Then a shortest tree connecting the vertices of the triangle is not longer than $3r$, but the minimal spanning tree is longer than $2\sqrt{3}\,r$. The proof is completed.
\end{proof}

\begin{ass}\label{ass:negative}
The infimum of the Steiner ratios over all two-dimensional polyhedra is equal to $1/2$. Similarly, the infimum of the Steiner ratios over all two-dimensional cones is equal to $1/2$.
\end{ass}

\begin{proof}
Consider a polyhedron $X$ containing a vertex with total angle $\pi k$, $k\ge 3$, and let $M$ be a regular $k$-gon inscribed into a circle centered at this vertex and having a sufficiently small radius $r$. Then  $\smt_X(M)$ does not exceed $k\,r$, but $\mst_X(M)=2(k-1)r$, therefore $\sr_X(M)\le(1/2)k/(k-1)\to1/2$ as $k\to\infty$. In the case of two-dimensional cones, instead of the polyhedron $X$ we consider a cone with total angle $\pi k$, $k\ge 3$ at its vertex. Assertion is proved.
\end{proof}

A triangular pyramid is called an  \emph{isosceles tetrahedron\/} or a \emph{disphenoid}, if all its faces are pairwise congruent. It is well-known that the surface of an isosceles tetrahedron can be locally isometric branched covered by the Euclidean plane, see~\cite{IPTTorus}. Applying Corollary~\ref{cor:cover-e} we obtain the following result.

\begin{cor}\label{cor:2-dim_tetr}
The Steiner ratio of the surface of an isosceles tetrahedron is equal to the Steiner ratio of the Euclidean plane.
\end{cor}

The action of the group $\Z_k$ on the Euclidean plane by rotations by the angles $2\pi/k$ around the origin defines the locally isometric branched covering of the cone with angle of $2\pi/k$ at the vertex by the Euclidean plane. Corollary~\ref{cor:cover-e} implies the following result.

\begin{cor}\label{cor:2-dim_cones}
The Steiner ratio of the two-dimensional cone, whose angle at the vertex is equal to $2\pi/k$, where $k$ is a positive integer, is equal to the Steiner ratio of the Euclidean plane.
\end{cor}

\section* {Acknowledgments}
\markright {\thesection.~Examples: Polyhedra and Cones}
The authors are grateful to Academician Anatoly Fomenko for permanent attention to their work. The work is partly supported by the Grant of President of RF for supporting of leading scientific schools of Russia, Project NSh--581.2014.1, and by RFBR, Project~13--01--00664a. A.~Ivanov takes the opportunity to thank Professor M.~Firer and Professor C.~Lavor for their kind invitation to the Conference ``Many Faces of Distances'' held in  UNICAMP, Campinas,  October 2014, and for their hospitality.

\markright {References}
\bibliographystyle{plain}

\begin{thebibliography}{99}

\bibitem{AD} 
A.\,D.~Alexandrov, 
{\it Convex Polyhedra}
(GITTL, Moscow, 1950; Springer-Verlag, Berlin, Heidelberg, 2005).

\bibitem{Cies} 
D.~Cieslik, 
\emph{The Steiner Ratio\/}  
(Springer-Verlag, Berlin, Heidelberg, New York,  2001).

\bibitem{CiesRep} D.~Cieslik, \emph{The Steiner Ratio --- A Report}, \newline{\tt stubber.math-inf.uni-greifswald.de/biomathematik/cieslik/massey.pdf}.

\bibitem{CiesIT} 
D.~Cieslik, A.\,O.~Ivanov, A.\,A.~Tuzhilin,
``Steiner Ratio for Manifolds'',  Mathematical Notes,  {\bf 74} (3), 367--374 (2003).

\bibitem{DuHwang} 
D.\,Z.~Du and F.\,K.~Hwang,
``A Proof of Gilbert--Pollak Conjecture on the Steiner ratio'', Algorithmica, {\bf 7}, 121--135 (1992).

\bibitem{GGJ} 
M.\,R.~Garey, R.\,L.~Graham, and D.\,S.~Johnson, 
``Some $NP$-complete Geometric Problems'',
in: {\it Eighth Annual Symp. on Theory of Comput.}, (1976) pp.~10--22.

\bibitem{GilPol} E.\,N.~Gilbert and H.\,O.~Pollak, ``Steiner Minimal Trees'',  SIAM J\. Appl\. Math.,  {\bf 16} (1), 1--29 (1968).

\bibitem{FomFuks}
A.\,T.~Fomenko, D.\,B.~Fuchs, V.\,L.~Gutenmacher,
{\it Homotopic Topology}
(State Mutual Book and Periodical Service,  1986).

\bibitem{Hw} F.\,K.~Hwang, ``On Steiner Minimal Trees with Rectilinear Distance'', SIAM J. Appl. Math., {\bf 30}, 104--114 (1976).

\bibitem{IK}
N.~Innami and B.\,H.~Kim,
``Steiner Ratio for Hyperbolic Surfaces'',
Proc. Japan Acad.,  Ser.~A, {\bf 82} (6),  77--79 (2006).

\bibitem{ITbookWP} A.\,O.~Ivanov and A.\,A.~Tuzhilin, {\it Branching Solutions to One-Dimensional Variational Problems}, (World Scientific Pul., Singapore, New Jersey, London, Hong Long, 2001).

\bibitem{ITJar} A.\,O.~Ivanov, A.\,A.~Tuzhilin, ``Optimal networks'',  in: {\it First Yaroslavl Summer School on Discrete and Computational Geometry}, (Izd-vo Yaroslavl Gos. Univ., Yaroslavl' 2013), pp.~60--110 [{\tt arXiv:1210.6228}] .

\bibitem{ITLup} A.\,O.~Ivanov, A.\,A.~Tuzhilin, ``Steiner Ratio. The State of the Art'', Matemat. Voprosy Kibern.,  {\bf 11}, 27--48 (2002) [in Russian].

\bibitem{ITAlg}
A.\,O.~Ivanov and A.\,A.~Tuzhilin,
``The Steiner Ratio Gilbert--Pollak Conjecture Is Still Open. Clarification Statement'',
Algorithmica,  {\bf 62} (1--2), 630--632 (2012).

\bibitem{IT22}
A.\,O.~Ivanov, A.\,A.~Tuzhilin,
``Du-Hwang Characteristic Area: Catch-22'', {\tt arXiv:1402.6079}.

\bibitem{ITZapMS}
A.\,O.~Ivanov, A.\,A.~Tuzhilin,
``One-Dimensional Gromov Minimal Filling Problem'',
Math. Sbornik, {\bf 203} (5), 65--118 (2012)
[Sbornik Mathematics, {\bf 203}  (5),  677--726 (2012)].

\bibitem{ITZapAMS}
A.\,O.~Ivanov, A.\,A.~Tuzhilin,
``Minimal Fillings of Finite Metric Spaces: The State of the Art'',  in {\it Discrete Geometry and Algebraic Combinatorics}, ed.: A.~Barg, and O.~Musin,  AMS Series Contemporary Mathematics, {\bf 625} (AMS, United States, 2014),  pp.~9--35.

\bibitem{IPTTorus} 
A.\,O.~Ivanov, I.\,V.~Ptitsyna, A.\,A.~Tuzhilin,
``Classification of Closed Minimal Networks on Flat Two-Dimensional Tori''
Matem. Sbornik, {\bf 183} (12), 3--44 (1992)
[Russian Academy of Sci., Sbornik. Mathematics,  {\bf 77} (2), 391--425 (1994)].

\bibitem{OvsComp}
Z.\,N.~Ovsyannikov,
``The Steiner and Gromov--Steiner Ratios and Steiner Subratio in the Space of Compacta in the Euclidean Plane with Hausdorff Distance'',
Fund. i Prikl. Matem., {\bf 18} (2), 157--165 (2013)
[J. of Math. Sciences, {\bf 203} (6), 858--863 (2014)].

\bibitem{Pakh}
A.\,S.~Pakhomova,
``A Continuity Criterion for Steiner-type Ratios in the Gromov--Hausdorff Space'',
Mat. Zametki, {\bf 96} (1), 126--137 (2014) 
[Math. Notes, {\bf 96} (1), 130--139 (2014)].

\bibitem{Zav}
E.\,A.~Zaval'nyuk,
``Steiner Ratio for the Hadamard Surfaces of Curvature at Most $k < 0$'',
Fund. i Prikl. Matem., {\bf 18} (2), 35--51(2013) [J. of Math. Sci., {\bf 203} (6) (2014)].


\bibitem{Zieg}
G.~Ziegler, {\it Lectures on Polytopes\/} 
(Springer-Verlag, New York, 1994).

\end{thebibliography}

 \end {document}